\documentclass[ECP]{ejpecp} 

\usepackage[T1]{fontenc}
\usepackage[utf8]{inputenc}
\usepackage{color}

\SHORTTITLE{Mixing trichotomy for an Ehrenfest urn with impurities}

\TITLE{Mixing trichotomy for an Ehrenfest urn with impurities}

\AUTHORS{%
Matteo~Quattropani\footnote{Dipartimento di Matematica ``Guido Castelnuovo'', Sapienza Universit\`a di Roma, Piazzale Aldo Moro 5, 00185 Roma, Italy.
    \BEMAIL{matteo.quattropani@uniroma1.it} }}
 
\KEYWORDS{mixing time of Markov chains; cutoff; Ehrenfest urn; negative dependence; mixing trichotomy} 

\AMSSUBJ{Primary 60K35; secondary  60J27; 82C26} 

\SUBMITTED{29 December 2023} 
\ACCEPTED{11 July 2024} 

\VOLUME{0}
\YEAR{2024}
\PAPERNUM{0}
\DOI{10.1214/YY-TN}

\ABSTRACT{We consider a version of the classical Ehrenfest urn model with two urns and two types of balls: regular and heavy. Each ball is selected independently according to a Poisson process having rate $1$ for regular balls and rate $\alpha\in(0,1)$ for heavy balls, and once a ball is selected, is placed in a urn uniformly at random. We study the asymptotic behavior when the total number of balls, $N$, goes to infinity, and the number of heavy ball is set to \textcolor{black}{$m_N\in\{1,\dots, N-1\}$}. We focus on the observable given by the total number of balls in the left urn, which converges to a binomial distribution of parameter $1/2$, regardless of the choice of the two parameters, $\alpha$ and \textcolor{black}{$m_N$}. We study the speed of convergence and show that this can exhibit three different phenomenologies depending on the choice of the two parameters of the model.}

\begin{document}


\section{Introduction}
The \emph{cutoff phenomenon}, first denoted with this name by Aldous and Diaconis \cite{AD86}, is one of the most intriguing and well studied topics in the probabilistic literature of the last 30 years. Indeed, despite the topic received much attention, and the list of Markov processes for which it is possible to show such a behavior is now quite long, a general sufficient condition ensuring the presence of cutoff is still missing. To this aim, it could be beneficial to reach a good understanding of  the mixing time of Markov chains that are obtained by perturbing an original chain that is known to exhibit cutoff. How does the mixing behavior of the modified chain depend on the perturbation? To which extent is the cutoff phenomenon robust? In the last few years, similar questions have been investigated in \cite{AGHH19,AGHHN22,CQrsa,CQspa}. Despite the models in the aforementioned papers are quite different, they can all be framed into a setting in which the perturbed chain is affected by the competition of two different mixing mechanisms: on the one hand, the abrupt convergence to equilibrium of the original chain, while, on the other hand, a smooth decay to equilibrium arising as an effect of the perturbation. Moreover, the models in \cite{AGHH19,AGHHN22,CQspa} share the common feature that the mixing stochastic process under investigation is not a Markov process itself, but rather a non-Markovian observable of a Markov process. A similar investigation, namely the analysis of mixing behavior of observables (or \emph{features} or \emph{statistics}) of Markov chains, has been recently performed in \cite{W19a,W19b} in the context of card shuffling routines and related models.

From the phenomenological point of view, the moral of the results in \cite{AGHH19,AGHHN22,CQrsa,CQspa} is that there exists a threshold for the perturbation parameter below which the mixing behavior of the process is unaffected. When the parameter exceeds the threshold, then the perturbation becomes dominant and the convergence to equilibrium takes place in a smooth fashion. Finally, and more interestingly, a non trivial mixing behavior is shown to take place at the interphase between these two regimes.

In the same spirit, this short note aims at analyzing the convergence to equilibrium of a two-parameter perturbation of one of the most classical examples of Markov chains exhibiting cutoff:  the Ehrenfest urn. Beyond the classical Ehrenfest urn model, \textcolor{black}{\cite{D96}}, the cutoff phenomenon is particularly well understood in the more general setting of birth-and-death chains \cite{CSC13,CSC15,DLP10,DSC06}. Therefore, it seems natural to consider perturbed versions of this classical model and check for the robustness of the cutoff phenomenon with respect to the perturbation. Contrarily to the works \cite{AGHH19,AGHHN22,CQrsa,CQspa}, the proof of our result does not require a strong technical machinery. Indeed, as it will be explained in Section \ref{sec:tv-ub}, the proof relies on a recently introduced approach to the convergence to equilibrium of spin systems, based on negative dependence, which has been developed by Salez \cite{S22} to prove cutoff for the (reversible) Simple Exclusion Process with reservoirs on arbitrary graphs. Such a technique is purely probabilistic, astonishingly simple, and perfectly suited for the analysis of our model.
\section{Model and results}
We consider the following generalization of the classical Ehrenfest urn model. There are $N$ balls divided in two urns. \textcolor{black}{A  number $m=m_N\in\{1,\dots,N-1 \}$ of} balls are \emph{heavy}, while the remaining $\textcolor{black}{n=n_N=N-m_N}$ balls are \emph{regular}. Each ball is selected at the arrival times of an independent Poisson process having rate $1$ if the ball is regular, and rate $\alpha\in(0,1\textcolor{black}{)}$ if the ball is heavy. When a ball is selected, it decides to which urn to move by tossing a fair coin, independently of the rest. Assuming that one can distinguish between regular and heavy balls, but that within the same class the balls are indistinguishable\textcolor{black}{, for any given $N$, $m$ and $\alpha$ as above, our model can be described as a Markov chain $(R_t,H_t)_{t\ge 0}$ with state space $\mathcal{X}\coloneqq\{0,1,\dots n \}\times\{0,1,\dots,m\}$, where the state $(a,b)\in \mathcal{X}$ stands for the configuration in which in the left urn there are $a$ regular balls and $b$ heavy balls.}
It is immediate to check that such a Markov chain is reversible with respect to the stationary distribution:
$$\chi_\star(a,b)\coloneqq\frac{\binom{n}{a}}{2^n} \frac{\binom{m}{b}}{2^m},\qquad (a,b)\in\mathcal{X}\ .$$
\textcolor{black}{In what follows, when the the initial configuration is $(r,h)\in \mathcal{X}$ i.e., $\mathbb{P}(R_0=r,H_0=h)=1$, we will denote by $\chi_t^{r,h}$ the distribution of $(R_t,H_t)$.}
We are interested in the limit $N\to \infty$, and we allow $\alpha$ \textcolor{black}{and $m$} to depend on $N$.
\begin{color}{black}
	In what follows, it will turn out to be useful to define 
	\begin{equation}\label{eq:def-beta}
		\beta=\beta_N=\frac{\log(m_N)}{\log(N)}\,.
	\end{equation}
\end{color}
As a first step, it is worth \textcolor{black}{pointing} out which is the mixing behavior of this Markov chain. To this scope, call
\begin{color}{black}
\begin{equation*}
	\widetilde{\mathcal{D}}_{\alpha,\textcolor{black}{m}}^{N,\rm tv}(t)\coloneqq\max_{(r,h)\in\mathcal{X}}\|\chi_t^{r,h}-\chi_\star \|_{\rm tv}=\frac12\max_{(r,h)\in\mathcal{X}}\sum_{(a,b)\in \mathcal{X}}\left|\chi^{r,h}_t(a,b)-\chi_\star(a,b) \right|\,,
\end{equation*}
\end{color}
\begin{color}{black}
and define
\begin{equation}
	t_{\rm mix}\coloneqq\inf\left\{t\ge 0\mid \widetilde{\mathcal{D}}^{N,{\rm tv}}_{\alpha, m}(t)\le \frac14 \right\}\,,
\end{equation}
where the dependence of $t_{\rm mix}$ on the parameters of the model is suppressed to ease the reading.
\end{color}
The next result is not hard to derive, and we postpone its proof to Section \ref{sec:proof-prop}.
\begin{proposition}\label{prop:whole-chain}
	For all $N\in\mathbb{N}$ fix $\alpha=\alpha_N\in(0,1\textcolor{black}{)}$, \textcolor{black}{$m=m_N\in\{1,\dots,N-1\}$, and recall the definition of $\beta=\beta_N$ in \eqref{eq:def-beta}}. 
	\begin{color}{black}
		Consider the sequence
		\begin{equation}\label{eq:def-gamma}
			\tilde\gamma=\tilde\gamma_N= \beta_N-\alpha_N\ ,
		\end{equation}
		and assume that
		$$\exists\: \tilde\gamma_\infty=\lim_{N\to\infty}\tilde\gamma_N\in[-1,+1]\ .$$
	\end{color}
	Depending of the asymptotic behavior of the parameters $\alpha=\alpha_N$ and $\beta=\beta_N$ three different mixing behaviors can take place: 
		\begin{itemize}
		\item \emph{Insensitivity:} if \,\textcolor{black}{ $\tilde{\gamma}_\infty<0$} there exists $C=C(\varepsilon)>0$ such that, called
		\begin{equation}\label{eq:def-t-R}
			t^{R}=\frac12\log(N)\,,
		\end{equation}
\textcolor{black}{we have}
		\begin{equation}\label{eq:prop-insensitivity}
			\liminf_{N\to\infty}	\widetilde{\mathcal{D}}_{\alpha,\textcolor{black}{m}}^{N,\rm tv}(t^{R}-C)\ge 1-\varepsilon\ ,\qquad \limsup_{N\to\infty}	\widetilde{\mathcal{D}}_{\alpha,\textcolor{black}{m}}^{N,\rm tv}(t^{R}+C)\le \varepsilon  \ .
		\end{equation}
		\item \emph{Delayed cutoff:}  if \,\textcolor{black}{ $\tilde{\gamma}_\infty\ge0$} and \textcolor{black}{$m\to\infty$}
		then, for every $\varepsilon>0$ there exists $C=C(\varepsilon)>0$ such that, called 
		\begin{equation}
			t^H=\frac\beta{2\alpha}\log(N)\,,
		\end{equation}
\textcolor{black}{we have}
		\begin{equation}\label{eq:prop-delayed}
			\liminf_{N\to\infty}	\widetilde{\mathcal{D}}_{\alpha,\textcolor{black}{m}}^{N,\rm tv}(t^{H}-C\alpha^{-1})\ge 1-\varepsilon\ ,\qquad	\limsup_{N\to\infty}	\widetilde{\mathcal{D}}_{\alpha,\textcolor{black}{m}}^{N,\rm tv}(t^{H}+C\alpha^{-1})\le \varepsilon  \ .
		\end{equation}
		\item \emph{No cutoff:} if \,\textcolor{black}{ $\tilde{\gamma}_\infty\ge0$} and \textcolor{black}{$\sup_Nm_N<\infty$},
		then for every constant $C>0$ sufficiently large there exists $0<\varepsilon_1(C)<\varepsilon_2(C)<1$ such that
		\begin{equation}\label{eq:prop-no-cutoff}
			\liminf_{N\to\infty}	\widetilde{\mathcal{D}}_{\alpha,\textcolor{black}{m}}^{N,\rm tv}(C\alpha^{-1})\ge \varepsilon_1(C),\qquad \limsup_{N\to\infty}	\widetilde{\mathcal{D}}_{\alpha,\textcolor{black}{m}}^{N,\rm tv}(C\alpha^{-1})\le \varepsilon_2(C)\ .
		\end{equation}
	\end{itemize}
\end{proposition}
\textcolor{black}{As a first observation, notice that, for a single class of balls, namely when considering the classical Ehrenfest urn, \textcolor{black}{it is} well-known \textcolor{black}{that the process exhibits} cutoff  at $\frac{t_{\rm rel}}2 \log(N)+\Theta(t_{\rm rel})$, see, e.g., \cite{D96}, when the number of balls goes to infinity.}
Therefore, Proposition \ref{prop:whole-chain} can be summarized by saying that the process exhibits cutoff at a time that is the largest among the mixing times of its two coordinates, and to disrupt this phenomenology it is necessary and sufficient to require that $\textcolor{black}{m}\asymp1$ (so that the ``heavy'' coordinate cannot exhibit cutoff) and that at the same time \textcolor{black}{$\beta\ge \alpha$.}

It is worth \textcolor{black}{pointing out} that, given that $(R_t,H_t)$ is a product chain, its relaxation time is $t_{\rm rel}=\alpha^{-1}$. \textcolor{black}{In light of this, it is not difficult to deduce the following result from Proposition \ref{prop:whole-chain}.
	\begin{corollary}\label{coro}
		In the setup of Proposition \ref{prop:whole-chain} there is cutoff if and only if the \emph{product condition}
		\begin{equation}\label{eq:prod-cond}
		\lim_{N\to\infty}	\frac{t_{\rm mix}}{t_{\rm rel}}=\infty\,,
		\end{equation}
	holds.
	\end{corollary}
} 
In fact, in the framework of the mixing time of Markov chains, it is known that the the \emph{product condition} \textcolor{black}{in \eqref{eq:prod-cond}} is necessary for cutoff \cite[Prop. 18.4]{LP17}, and it has been shown to be sufficient in specific setups, such as birth-and-death chains \cite{DLP10}, random walks on trees \cite{BHP17} and the reversible simple exclusion process with reservoirs \cite{S22}. 

As mentioned in the Introduction, in this note we aim at understanding, rather than the mixing behavior of the chain, the mixing time of a specific statistic.  
More precisely, we interpret the \emph{heavy balls} as \emph{impurities}, and we assume that the observer is not aware of their existence. In other words, the observer can only count the number of balls in the two containers, but cannot distinguish between heavy and regular balls. \textcolor{black}{In absence of impurities, by looking at such a statistic for the process at a time slightly larger than $t^R$ for several independent experiments, the observer should see configurations ``statistically indistinguishable from  equilibrium''.
The question we are going to address is: how long does it take for the observer to see configurations ``statistically indistinguishable from  equilibrium'' in presence of impurities?} 
\textcolor{black}{More precisely, for every initial configuration $(r,h)\in\mathcal{X}$ and $t\ge 0$ we consider the quantity $W_t\coloneqq R_t+H_t$ and call $\nu_t^{r,h}\coloneqq{\rm Law}(W_t)$.}
We aim at investigating how this convergence takes place as a function of the two parameters $(\alpha,\textcolor{black}{m})$ in the limit in which the number of balls goes to infinity. To this aim, for every fixed $N$ and choice of the parameters $\alpha=\alpha_N$ and $\textcolor{black}{m}=\textcolor{black}{m_N}$\textcolor{black}{,} we will consider the function
$$t\mapsto \mathcal{D}_{\alpha,\textcolor{black}{m}}^{N,{\rm tv}}(t)\coloneqq \max_{(r,h)\in \mathcal{X}}\left\|\nu^{r,h}_t-\nu_\star\right\|_{\rm tv}=\frac12 \max_{(r,h)\in \mathcal{X}}\sum_{k=0}^{N}|\nu^{r,h}_t(k)-\nu_\star(k)| \ ,$$
where $\nu_\star$ is the law of a Binomial random variable of parameters $N$ and $\tfrac12$.
\begin{theorem}\label{th:tv}
	Consider the sequence
	\begin{equation}\label{eq:def-gamma}
		\gamma=\gamma_N= \frac{2\beta-1}{\alpha}-1\ ,
	\end{equation}
	and assume that
	$$\exists\: \gamma_\infty=\lim_{N\to\infty}\gamma_N\in[-\infty,+\infty]\ .$$
	Depending of the asymptotic behavior of the parameters $\alpha=\alpha_N$ and $\beta=\beta_N$ three different mixing behaviors can take place: 
	\begin{itemize}
		\item \emph{Insensitivity:} if $\gamma_\infty<0$  then, for every $\varepsilon>0$ there exists $C=C(\varepsilon)>0$ such that,
		\begin{equation}
			\liminf_{N\to\infty}\mathcal{D}_{\alpha,\textcolor{black}{m}}^{N,{\rm tv}}(t^{R}-C)\ge 1-\varepsilon\ ,\qquad \limsup_{N\to\infty}\mathcal{D}_{\alpha,\textcolor{black}{m}}^{N,{\rm tv}}(t^{R}+C)\le \varepsilon  \ ,
		\end{equation}
	where $t^R$ is as in \eqref{eq:def-t-R}.
		\item \emph{Delayed cutoff:} if $\gamma_\infty\ge 0$ 
		and 
		$(2\beta-1)\log(N)\to\infty$, 
		then, for every $\varepsilon>0$ there exists $C=C(\varepsilon)>0$ such that, called
		$$t^{\rm dc}=\frac{1+\gamma}2 \log(N) \ ,$$
		we have
		\begin{equation}
			\liminf_{N\to\infty}\mathcal{D}_{\alpha,\textcolor{black}{m}}^{N,{\rm tv}}(t^{\rm dc}-C\alpha^{-1})\ge 1-\varepsilon\ ,\qquad	\limsup_{N\to\infty}\mathcal{D}_{\alpha,\textcolor{black}{m}}^{N,{\rm tv}}(t^{\rm dc}+C\alpha^{-1})\le \varepsilon  \ .
		\end{equation}
		\item \emph{No cutoff:} if $\gamma_\infty\ge 0$  and 
		$(2\beta-1)\log(N)\to\ell\in[0,\infty)$,
		then for every constant $C=C(\ell)>0$ sufficiently large there exists $0<\varepsilon_1(C)<\varepsilon_2(C)<1$ such that
		\begin{equation}\label{eq:th-no-cutoff}
			\liminf_{N\to\infty}\mathcal{D}_{\alpha,\textcolor{black}{m}}^{N,{\rm tv}}(C\alpha^{-1})\ge \varepsilon_1(C),\qquad \limsup_{N\to\infty}\mathcal{D}_{\alpha,\textcolor{black}{m}}^{N,{\rm tv}}(C\alpha^{-1})\le \varepsilon_2(C)\ .
		\end{equation}
	\end{itemize}
\end{theorem}

Let us now briefly comment on the result. We might define the phase $\gamma_\infty<0$ as \emph{subcritical}: the strength of the perturbation is not enough to affect the mixing behavior of the process. Notice that in this case, letting $\alpha_N\to 0$ arbitrarily fast, essentially constraining the heavy particles on the \textcolor{black}{right} urn for an arbitrarily long time, Theorem \ref{th:tv} states that at least order $\sqrt{N}$ \emph{impurities} are necessary to affect the mixing mechanism. This is not at all surprising. Indeed, the positions of $o(\sqrt{N})$ balls is irrelevant due to the size of the fluctuations of the equilibrium measure. 

At the \emph{critical line} $\gamma_\infty=0$ a smooth transition takes place: indeed, one should notice that, if $\alpha=\Theta(1)$ and $\gamma^{-1}\gg\log N$, then we get again the same phenomenology as in the \emph{subcritical phase}. Nevertheless, as $\gamma$ grows, the cutoff time is delayed by the amount $\tfrac{\gamma}{2}\log N$ and, if $\alpha\to 0$, the cutoff window is enlarged by a factor $\alpha^{-1}$. It is enlightening to look at such a behavior from the following perspective: let $t_{\rm rel}=\alpha^{-1}$ be the relaxation time of the (whole) Markov process, so that it is natural to measure time on the timescale of $t_{\rm rel}$. Doing so, if $\gamma_\infty\ge 0$, Theorem \ref{th:tv} tells us that
\begin{equation}
	t_{\rm mix}^W(\varepsilon):=\inf\{t\ge 0\mid \mathcal{D}_{\alpha,\textcolor{black}{m}}^{N,{\rm tv}}(t)\le \varepsilon \}=t_{\rm rel}\big[(\beta-\tfrac{1}{2})\log(N)+ \Theta(1) \big]\,,\qquad \varepsilon\in(0,1) \ .
\end{equation}	
The latter equation makes clear the transition between the \emph{delayed cutoff phase} to the one in which the cutoff phenomenon is disrupted, as the finiteness of the limit of $(\beta-\tfrac{1}{2})\log N$ as $N\to\infty$ shows. In other words, a sort of \emph{product condition} is necessary and sufficient for the presence of cutoff.

\begin{remark}
It is worth \textcolor{black}{making} an explicit comparison between the regimes in Proposition \ref{prop:whole-chain} and those in Theorem \ref{th:tv}.
If \textcolor{black}{$\tilde\gamma_{\infty}<0$} then $\gamma_\infty<0$, so that not only the statistics of interest, but actually the whole chain, mixes at time $t^R$. On the other hand, if $\gamma_\infty\ge0$ then \textcolor{black}{$\tilde\gamma_{\infty}\ge0$ and $m\to\infty$, therefore in this case we have that the whole chain exhibits cutoff at $t^H$ while our statistic either mixes without cutoff on the timescale $\alpha^{-1}\ll t^H$, if $(2\beta-1)\log N\not\to\infty$, or it has cutoff at time $t^{\rm dc}\le t^H$ where the difference grows to infinity if $\lim_{N\to\infty}2\beta-1>0$.}
The remaining case is the one in which \textcolor{black}{$\tilde{\gamma}_\infty\ge0$ but $\gamma_{\infty}<0$, i.e., }$\alpha\le\beta<\frac{1+\alpha}{2}$. \textcolor{black}{In this case, if $m\to\infty$ then both the whole chain and our statistic exhibit cutoff at times $t^H$ and $t^R$, respectively, where the former is always larger than the latter and their ratio goes to infinity if $\beta\gg\alpha$ (in particular, we must have $\alpha\ll 1$). Finally, if $\tilde{\gamma}_\infty\ge0$, $\gamma_{\infty}<0$ but $m\asymp 1$, then the whole chain and our statistic reach equilibrium on scale $\alpha^{-1}$ (smoothly) and $t^{R}$ (with cutoff), respectively, and the ratio between the two scales goes to infinity if and only if $\alpha\ll (\log N)^{-1}$.}
\end{remark}

\begin{color}{black}
	\subsection{The coupling}\label{sec:representation-new}
	Despite the fact that the stochastic process we aim at investigating is not a Markov process, one may try to adapt the techniques used to analyze the classical Ehrenfest urn model. Like in \cite[Sec. 6.5.2]{LP17}, a natural approach to bound from above the total variation distance at time $t$ consists in finding a coupling of the laws $\nu_t$ and $\nu_\star$ and control the probability that the coupling fails. \textcolor{black}{It is well-know that in the classical Ehrenfest urn model the ``natural coupling'' (see \cite[Sec. 6.5.2]{LP17}) does not suffices} to obtain the exact constant of the mixing time, but it provides an estimate that is off by a factor $2$. We will show that also for our model \textcolor{black}{an approach in the same spirit} does not provide the right prefactor. With this aim in mind we define, for an arbitrary $t\ge 0$ the following jointly independent sequences of random variables
		\begin{equation}\label{eq:definitions-new}
			\begin{split}
					&(X^\star_i)_{\textcolor{black}{1\le i \le N}}\overset{d}{=}\bigotimes_{i=1}^n{\rm Bern}(\tfrac12)\\
			(\tilde Z_i(t))_{\textcolor{black}{1\le i \le N-m}}\overset{d}{=}\bigotimes_{i=1}^{N-m}{\rm Bern}&(\textcolor{black}{e^{-\alpha t}})\,,\qquad
			(\tilde Z_i(t))_{\textcolor{black}{N-m< i \le N}}\overset{d}{=}\bigotimes_{i=N-m+1}^N{\rm Bern}(\textcolor{black}{e^{-t}})\ ,
					\end{split}
	\end{equation}
	and for any prescribed initial position $(r_0,h_0)\in\{0,\dots,N-m\}\times\{0,\dots,m\}$ set $y=(y_1,\dots,y_N)$ where
	\begin{equation}
		y_i=\begin{cases}
			1&\text{if }i\le r_0\text{ or }i> N-h_0\\
			0&\text{otherwise}
		\end{cases}\,.
	\end{equation}
	Now define
	\begin{equation}\label{eq:tildeX}
		\tilde X_i(t)=(1-\tilde Z_i(t))X_i^\star+\tilde Z_i(t) y_i\,, \qquad i\in[N]\,,
	\end{equation}
	and notice that 
	\begin{equation}
		\left(\sum_{i=1}^{N-m}\tilde X_i(t),\sum_{i=N-m+1}^N \tilde{X}_i(t)\right)\overset{d}{=}\big(R_t,H_t\big)\,,
	\end{equation}
	and, as a consequence,
	\begin{equation}\label{eq:215-1}
\sum_{i=1}^N \tilde{X}_i(t)\sim \nu_t^{r_0,h_0}\,.
	\end{equation}
	Indeed, the random variable $\tilde Z_i(t)$ stands for the event that the $i$-ball has never been selected up to time $t$ and, as modeled by \eqref{eq:tildeX}, if $\tilde Z_i(t)=0$ the $i$-th ball in uniformly distributed over the two urns (regardless of its initial position $y_i$).
	Under such an explicit coupled construction, we consider the event $$\mathcal{W}_t\coloneqq\big\{\textstyle\sum_{i=1}^N\tilde Z_i(t)=0\big\}\,.$$ 
	By the properties of the total variation distance and by Markov inequality,
	\begin{equation}\label{eq:ub-sep-new}
		\mathcal{D}_{\alpha,\textcolor{black}{m}}^{N,{\rm tv}}(t)\le \mathbb{P}(\mathcal{W}_t^c)=\mathbb{P}\left(\sum_{i=1}^{N}\tilde Z_i(t)>0\right)\le \mathbb{E}\bigg[\sum_{i=1}^N\tilde Z_i(t)\bigg]\ .
	\end{equation}
	Since,
	\begin{equation}
		\mathbb{E}\bigg[\sum_{i=1}^N\tilde Z_i(t)\bigg]= me^{-\alpha t}+(N-m)e^{-t}\ ,
	\end{equation}
	by plugging in the values of $t=t^{R}+O(1)$ or $t=t^{\rm dc}+O(\alpha^{-1})$ we see that the upper bound in \eqref{eq:ub-sep-new} is too weak to provide the sharp estimates in Theorem \ref{th:tv}.
	
	To get a better upper bound, the key observation lies in defining another coupling as follows. Let $\sigma$ be a uniformly random permutation on $[N]$, and call
	\begin{equation}
		(Z_i,Y_i)=(\tilde Z_{\sigma(i)}, y_{\sigma(i)})\,,\qquad i\in[N]\,,
	\end{equation}
	and
	\begin{equation}\label{eq:defX}
		X_i(t)=(1-Z_i(t))X_i^\star+Z_i(t) Y_i\,,\qquad i\in[N]\,.
	\end{equation}
	Since $X_i^\star\overset{d}{=}X_j^\star$ for all $i,j\in\{1,\dots, N\}$ and the random variables $(X_i^\star)_{1\le i\le N}$, $(\tilde Z_i(t))_{1\le i\le N}$ and $\sigma$ are independent, we deduce that
	\begin{equation}\label{eq:215-2}
		(X_i(t))_{1\le i\le N}\overset{d}{=}\big((1-\tilde Z_{\sigma(i)}(t))X^\star_{\sigma(i)}+\tilde{Z}_{\sigma(i)}(t)y_{\sigma(i)} \big)_{1\le i\le N}=(\tilde X_{\sigma(i)}(t))_{1\le i \le N}\,,
	\end{equation}
	and, by \eqref{eq:215-1} and \eqref{eq:215-2}, we conclude that
	\begin{equation}
	\sum_{i=1}^N {X}_i(t)\sim \nu_t^{r_0,h_0}\,.
	\end{equation}
	In what follows it will be convenient to work with a more direct representation of the sequence of random variables $(Z_i)_{1\le i\le N}$. Notice indeed that, defined the jointly independent sequences of random variables
	\begin{equation}\label{eq:definitions-PQ}
		\begin{split}
			&(P_i(t))_{\textcolor{black}{1\le i \le N}}\bigotimes_{i=1}^N{\rm Bern}(e^{-\alpha t})\ ,\qquad
			(Q_i(t))_{\textcolor{black}{1\le i \le N}}\bigotimes_{i=1}^N{\rm Bern}(e^{-t})\ ,\\
			 &(V_i)_{\textcolor{black}{1\le i \le N}}\overset{\rm d}{=}{\rm Unif}\big(\{v\in\{0,1\}^N\mid {\textstyle \sum_{i}}v_i=m \}\big)\ ,
		\end{split}
	\end{equation}
	we have
	\begin{equation}\label{eq:representation}
		(Z_i(t))_{1\le i\le N}\overset{d}{=}\big(V_iP_i(t)+(1-V_i)Q_i(t)\big)_{1\le i\le N}\,.
	\end{equation}
\end{color}

\section{Strategy of proof}
The strategy of proof is quite simple, and crucially relies on the recent result presented in \cite{S22}, which introduces a control on the $L^2(\pi)$ distance of a negatively dependent perturbation of a product measure. More precisely, the result we are going to exploit reads as follows:
\begin{lemma}\label{lemma:S22}\textcolor{black}{\cite[Lemma 1]{S22}}. Let $(X^\star_i)_{\textcolor{black}{1\le i \le N}}$ be as in \eqref{eq:definitions-new}, $(Z_i)_{\textcolor{black}{1\le i \le N}}$ arbitrarily distributed on $\{0,1\}^N$ but \emph{negatively dependent}, i.e.,
	\begin{equation}\label{eq:neg-dependence}
		\mathbb{E}\left[\prod_{i\in A}Z_i(t) \right]\le \prod_{i\in A}\mathbb{E}[Z_i(t)],\qquad \forall A\subseteq [N]\ ,
	\end{equation}
\textcolor{black}{and independent of $(X^\star_i)_{1\le i\le N}$.} 
	Then, called $\pi$ the law of $(X_i^\star)_{\textcolor{black}{1\le i \le N}}$ and $\mu$ the law of $(X_i)_{\textcolor{black}{1\le i \le N}}$, defined as
\begin{color}{black}
	$$X_i=(1-Z_i)X_i^\star+Z_iY_i,\qquad i=1,\dots,N\ ,$$
	for an arbitrary $(Y_i)_{1\le i\le N}\in\{0,1\}^N$,
	we have
\end{color}
	\begin{equation}\label{eq:lemmaS22}
		\left\| \frac{\mu}{\pi}-1\right\|_{L^2(\pi)}^2 
		\le \prod_{i=1}^N(1+\mathbb{E}[Z_i]^2)-1 \ .
	\end{equation}
	Moreover, if  the one in \eqref{eq:neg-dependence} is an equality, so is the one in \eqref{eq:lemmaS22}.
\end{lemma}
Lemma \ref{lemma:S22} finds an ancestor in the exponential moment bound presented in \cite[Proposition 3.2]{MP12} (see also \cite[Lemma 3.1]{LS17}), and essentially shows that \emph{negative dependence} can be translated into a bound for the $L^2$ distance as the one in \eqref{eq:lemmaS22}, in a sharp sense.  As we will show in Section \ref{sec:tv-ub}, for every $t\ge 0$ the vector $(Z_i(t))_{\textcolor{black}{1\le i \le N}}$ in \eqref{eq:representation} is negatively dependent, hence providing a control on the total variation distance between $(X_i^\star)_{\textcolor{black}{1\le i \le N}}$ and $(X_i(t))_{\textcolor{black}{1\le i \le N}}$, so that the same control on the marginals given by the respective sums follows immediately by projection. As we will see in Section \ref{sec:tv-lb}, \textcolor{black}{the proof of the lower bound} is even simpler and relies on controlling the probability that $S^\star$ and $S(t)$ lie in the subinterval $\{0,1,\dots, \tfrac{N}{2}-k\}$, \textcolor{black}{for certain values of $k\in\{0,\dots, \tfrac{N}2\}$}. This is a particularly simple example of \textcolor{black}{what is usually referred to as} \emph{distinguishing statistic}, where the function considered is the indicator of the above mentioned intervals. It is well-known, see e.g. \cite[Sec. 7.3.1]{LP17}, that such a simple technique provides a sharp estimate for the mixing time in the classical Ehrenfest urn model. As we will show, this technique is effective also in our modified setting.
\subsection{Upper bound via Negative Dependence}\label{sec:tv-ub}
For notational simplicity, let us denote
$$\mu_t={\rm Law}(X_i(t),\:i\le N),\qquad\pi={\rm Law}(X_i^\star,\:i\le N)=\bigotimes_{i=1}^N{\rm Bern}(\tfrac12)\ .$$
Since both $S(t)$ and $S^\star$ are statistics of the laws $\mu_t$ and $\pi$, respectively, it is possible to bound from above 
\begin{equation}
\mathcal{D}_{\alpha,\textcolor{black}{m}}^{N,{\rm tv}}(t)\le \| \mu_t-\pi\|_{\rm tv}\ .
\end{equation}
We will show that for every given $t\ge 0$ the collection of random variables $(Z_i(t))_{\textcolor{black}{1\le i \le N}}$ defined in \eqref{eq:representation} is \emph{negatively dependent}, i.e., \eqref{eq:neg-dependence} holds.
Thanks to this fact, using the $L^2$ bound we get
\begin{equation}\label{eq:upper-bound}
	\begin{split}
		\| \mu_t-\pi\|_{\rm tv}\le\,\frac12\left\| \frac{\mu_t}{\pi}-1\right\|_{L^2(\pi)}
		\le&\:\frac12 \sqrt{\left[1+\frac14\left(\frac{m}{N}e^{-\alpha t}+\frac{N-m}{N}e^{-t}\right)^2\right]^N-1} \\
	\le&\,\textcolor{black}{\frac12 \sqrt{\left[1+\frac14\left(N^{\beta-1}e^{-\alpha t}+Ne^{-t}\right)^2\right]^N-1}\,.}
	\end{split}
\end{equation}
where the last inequality follows by Lemma \ref{lemma:S22} and \eqref{eq:mean-var-new}.
\begin{lemma}\label{lemma:neg-dep}
	For every $N\in\mathbb{N}$, $t\ge 0$ and $\alpha_N,\beta_N\in[0,1]$ the random vector $(Z_i(t))_{\textcolor{black}{1\le i \le N}}$ in \eqref{eq:representation}  satisfies \eqref{eq:neg-dependence}.
\end{lemma}
\begin{proof}
Fix $A\subseteq[N]$. Called \textcolor{black}{$u_t=e^{(1-\alpha) t}\ge 1$}, we have
	\begin{align*}
		\mathbb{E}[Z_i(t)]^{|A|}&=\bigg(\frac{m}{N}\textcolor{black}{e^{-\alpha t}}+\frac{N-m}{N}\textcolor{black}{e^{-t}}\bigg)^{|A|}\\
		&=\frac{1}{N^{|A|}}\sum_{a=0}^{|A|} \binom{|A|}{a}m^a (N-m)^{|A|-a}\textcolor{black}{\left(e^{-\alpha t}\right)}^a\textcolor{black}{\left(e^{-t}\right)}^{|A|-a}
		\\
		&=\textcolor{black}{\left(e^{-t}\right)}^{|A|}\mathbb{E}\left[u_t^B\right],\qquad B\sim{\rm Bin}(|A|,\tfrac{m}{N})\ .
	\end{align*}
	For the expectation on the left-hand-side of \eqref{eq:neg-dependence} we have
	\begin{equation}\label{hyper}
		\begin{split}
		\mathbb{E}\left[\prod_{i\in A}Z_i(t)\right]
		&=\frac{1}{(N)_{|A|}}\sum_{a=0}^{|A|} \binom{|A|}{a}(m)_a (N-m)_{|A|-a}\textcolor{black}{\left(e^{-\alpha t}\right)}^a\textcolor{black}{\left(e^{-t}\right)}^{|A|-a}\\
		&=\sum_{a=0}^{|A|}\frac{\binom{m}{a}\binom{N-m}{|A|-a}}{\binom{N}{|A|}}\textcolor{black}{\left(e^{-\alpha t}\right)}^a\textcolor{black}{\left(e^{-t}\right)}^{|A|-a}\\
		&=\textcolor{black}{\left(e^{-t} \right)}^{|A|}\mathbb{E}\left[u_t^{H}\right],\qquad H\sim{\rm Hypergeom}(N,m,|A|)\ ,
		\end{split}
	\end{equation}
	where $(x)_k$ denotes the falling factorial $x(x-1)\cdots(x-k+1)$. 
	\begin{color}{black}
	Indeed, the expectation on the left-hand side of \eqref{hyper} can be computed by conditioning on the number of elements in $i\in A$ such that $V_i=1$. For every $a\in\{0,\dots,|A|\}$ the probability that $|\{i\in A\mid V_i=1 \}|=a$ is given by
	\begin{equation*}
		\frac{m\dots (m-a+1)\cdot (N-m)\dots (N-m-|A|+a+1) }{N(N-1)\cdots (N-|A|+1)}=\frac{(m)_a(N-m)_{|A|-a}}{(N)_A}\,,
	\end{equation*}
	and, conditionally on this event, for each $\bar i\in\{i\in A\mid V_i=1 \}$ the expectation of $Z_{\bar i}(t)$ is $e^{-\alpha t}$ (and similarly for $\bar i\in\{i\in A\mid V_i=0\}$).
	\end{color}
	 Therefore, to conclude the proof it is sufficient to show that for every $u\ge 1$ we have
	\begin{equation}\label{eq:MGF}
		\mathbb{E}\left[u^B\right]\ge \mathbb{E}\left[u^{H}\right]\ .
	\end{equation}
	Notice that for every $k\ge 1$ and $D=B,H$
	\begin{align}\label{eq:derivative}
		\frac{\rm d^k}{{\rm d}u^k}\mathbb{E}[u^D]\big\rvert_{u=1}=\mathbb{E}[(D)_ku^{D-k}]\big\rvert_{u=1}=\mathbb{E}[(D)_k]\ .
	\end{align}
	Moreover, \textcolor{black}{it is known that the factorial moments of $B$ and $H$ are given by (see, e.g., \cite{P53})}
	\begin{equation}\label{eq:comparison-fact-mom}
		\mathbb{E}[(B)_k]=(|A|)_k\frac{m^k}{N^k}\ge (|A|)_k\frac{(m)_k}{(N)_k}=\mathbb{E}[(H)_k]\ .
	\end{equation}
	Therefore \eqref{eq:MGF} follows by a finite (since for $k>|A|$ the coefficients are zero) Taylor expansion of  $\mathbb{E}\left[u^B\right]$ and $\mathbb{E}\left[u^{H}\right]$ around $u=1$ noticing that, thanks to \eqref{eq:derivative} and \eqref{eq:comparison-fact-mom}, \textcolor{black}{the coefficients of the former polynomials} are larger than or equal to those of the latter.
\end{proof}

\subsection{TV lower bound via distinguishing statistics}\label{sec:tv-lb}
\begin{color}{black}
In this section we will focus on the initial condition $(r,h)=(0,0)$, and we will call
\begin{equation}
	S(t)\coloneqq \sum_{i=1}^t X_i(t)\,,\qquad S^\star\coloneqq\sum_{i=1}^{t}X^\star_i\,,
\end{equation}
where the sequences $(X^\star_i)_{1\le i\le N}$ and $(X_i(t))_{1\le i\le N}$ are defined in \eqref{eq:definitions-new} and \eqref{eq:defX}, respectively. Notice that, with this initial condition and introducing the convenient notation
\begin{equation}
	p_t\coloneqq\frac12 (1-e^{-\alpha t})\,,\qquad q_t\coloneqq \frac12 (1-e^{-t})\,,
\end{equation}
 we get
\begin{equation}\label{eq:mean-var-new}
	\begin{split}
		\mathbb{E}[S(t)]=&\: mp_t+(\textcolor{black}{N}-m)q_t\ ,\\
		{\rm Var}(S(t))=&\: m p_t(1-p_t)+(\textcolor{black}{N}-m)q_t(1-q_t)\le \frac{\textcolor{black}{N}}{4}\ ,\\
		\sum_{i=1}^N \mathbb{E}[Z_i(t)]=&\: N\mathbb{E}[Z_1(t)]=\textcolor{black}{N-2\mathbb{E}[S(t)]}\ .
	\end{split}
\end{equation}
\end{color}
The lower bounds are proved by considering the events $\mathcal{E}^-_k\coloneqq\{0,1,\dots,N/2-k \}$, 
and, for appropriate choices of $t$ and $k$, we use the lower bound
\begin{equation}\label{eq:idea-lb}
	\mathcal{D}_{\alpha,\textcolor{black}{m}}^{N,{\rm tv}}(t)\ge \mathbb{P}(S(t) \in \mathcal{E}^-_k)-\mathbb{P}(S^\star \in \mathcal{E}^-_k)\textcolor{black}{\,= 1-\mathbb{P}(S(t) \not\in \mathcal{E}^-_k)-\mathbb{P}(S^\star \in \mathcal{E}^-_k)}\ .
\end{equation} 
\begin{color}{black}
	On the one hand, by Chebyshev inequality we have
\begin{equation}\label{eq:lower-bound-CLT-star-k}
	\mathbb{P}(S^\star\in\mathcal{E}^-_k)\le \mathbb{P}(|S^\star-\mathbb{E}[S^\star]|\ge k)\le \frac{N}{4k^2}\,,
\end{equation}
so that, choosing $k=\lceil \tfrac{c}{2}\sqrt{N}\rceil$, we obtain
\begin{equation}\label{eq:lower-bound-CLT-star}
	\mathbb{P}(S^\star\in\mathcal{E}^-_k)\le \frac{1}{c^2}\,.
\end{equation}
	On the other hand, by \eqref{eq:mean-var-new} and Chebyshev inequality
	\begin{equation}\label{eq:lower-bound-CLT-t}
		\begin{split}
			\mathbb{P}(S(t)\not\in\mathcal{E}_k^-)&=\mathbb{P}\left(S(t)-\mathbb{E}[S(t)]\ge \frac{N}{2}-\mathbb{E}[S(t)]-k+1\right)\\
			&\le \frac{N}{4}\left(\frac{N}{2}-\mathbb{E}[S(t)]-k+1 \right)^{-2}\,,
		\end{split}
	\end{equation}
	hence, \textcolor{black}{if for some $c>0$},
	\begin{equation}\label{eq:check}
		\frac{N}{2}-\mathbb{E}[S(t)]\ge \textcolor{black}{c}\, \sqrt{N}
	\end{equation}
	choosing $k=\lceil\tfrac{ \textcolor{black}{c}}{2}\sqrt{N}\rceil$  we obtain
	\begin{equation}\label{eq:lower-bound-CLT-t-k}
		\begin{split}
			\mathbb{P}(S(t)\not\in\mathcal{E}_k^-)\le\,\frac{1}{c^2}\,.
		\end{split}
	\end{equation}
\subsubsection{Refined lower bounds via CLT}\label{suse:CLT}
For the \emph{no cutoff} case in Theorem \ref{th:tv} we will need a refined estimate which consists in bounding the probabilities on the right-hand side of \eqref{eq:idea-lb} using the CLT rather than Chebyshev inequality. To do so, notice that
\end{color}
\begin{equation}
	\mathbb{P}(S^\star\in\mathcal{E}^-_k)=\mathbb{P}\left(2\sqrt{N}\left( \frac1N S^\star-\frac12\right)\le-\frac{2k}{\sqrt{N}}\right)\ ,
\end{equation}
so that, choosing $k=\lceil\tfrac{c}{2}\sqrt{N}\rceil$ for some $c>0$, by the Central Limit Theorem we have
\begin{equation}\label{eq:lower-bound-CLT-star-bis}
	\mathbb{P}(S^\star\in\mathcal{E}^-_k)\to\Phi(-c)\,,\qquad\text{ as \ } N\to\infty\ .
\end{equation}
Similarly,
\begin{equation}
	\begin{split}
		\mathbb{P}(S(t)&\not\in\mathcal{E}_k^-)=\\
		&\mathbb{P}\left(\frac{N}{\sqrt{{\rm Var}(S(t))}}\left(\frac{S(t)}{N}-\frac{\mathbb{E}[S(t)]}{N} \right)> \frac{1}{\sqrt{{\rm Var}(S(t))}}\left(\frac{N}{2}-\mathbb{E}[S(t)]-k\right)\right)\ ,
	\end{split}
\end{equation}
hence, if for some $c>0$
\begin{equation}
	\frac{N}{2}-\mathbb{E}[S(t)]\ge c \sqrt{N}
\end{equation}
choosing $k=\lceil\tfrac{c}{2}\sqrt{N}\rceil$ we obtain, for $t=t_N\to\infty$,
\begin{equation}\label{eq:lower-bound-CLT-t-k-bis}
	\begin{split}
		\mathbb{P}(S(t)\not\in\mathcal{E}_k^-)\le&\:\mathbb{P}\left(\frac{N}{\sqrt{{\rm Var}(S(t))}}\left(\frac{1}{N}S(t)-\mathbb{E}[S(t)] \right)> \frac{c}2\frac{\sqrt{N}}{\sqrt{{\rm Var}(S(t))}}\right)\\
		\to&\:1-\Phi\left(c \right)\ ,
	\end{split}
\end{equation}
where we used that $t=t_N\to\infty$ implies
$\sqrt{{\rm Var}(S(t)) N^{-1}}\to \frac12$, and in the last inequality that $N$ is sufficiently large.

\section{Proof of Theorem \ref{th:tv}}\label{sec:proofs}
We start by considering the case in which $\gamma_\infty<0$ and therefore, for every $N$ large enough, $\beta<\tfrac12(1+\alpha)$. 
\begin{proof}[Insensitivity]
	We start by proving the upper bound. Let $t^+=\tfrac12\log N + C$, for some $C>0$ to be determined later. Then, by \eqref{eq:upper-bound} we have, for all $N$ sufficiently large
	\begin{align*}
		\mathcal{D}_{\alpha,\textcolor{black}{m}}^{N,{\rm tv}}(t^+)\le& \frac12 \sqrt{\left[1+\frac14\left(N^{\beta-1-\frac\alpha{2}}e^{-C\alpha}+N^{-\frac12}e^{-C}\right)^2\right]^N-1} \\
		\le&\frac12 \sqrt{\left[1+\frac{e^{-2C}}{N}\right]^N-1} \le\sqrt{e^{e^{-2C}}-1}\ ,
	\end{align*}
	and the latter can be made arbitrarily small by taking $C\to\infty$. We now prove the lower bound and set $t^-=\tfrac12\log N - C$. 
	By \eqref{eq:mean-var-new} we have
	\begin{equation}\label{eq:eps}
		\frac{N}{2}-\mathbb{E}[S(t^-)]\textcolor{black}{\,=\frac{1}{2} \left(m N^{-\frac{\alpha}{2}}e^{\alpha C}+(N-m)N^{-\frac12}e^C \right)\ge \frac14\sqrt{N} e^C}\ ,
	\end{equation}
\textcolor{black}{where for the inequality we only used that if $\gamma_{\infty}<0$ then $\beta<1$, thus $m\le N/2$ for all $N$ large enough.}
	Hence, \eqref{eq:check} holds with $ \textcolor{black}{c}= \frac14 e^C$. Choosing $k=\lceil\frac{ \textcolor{black}{c}}{2}\sqrt{N}\rceil$, thanks to \eqref{eq:lower-bound-CLT-t-k} and \eqref{eq:lower-bound-CLT-star} we have 
	\begin{equation}\label{eq:lower-bound-subcritical}
		\mathbb{P}(S(t^-)\not\in\mathcal{E}_k^-)\le 16e^{-2C}\,,\qquad \mathbb{P}(S^\star\in\mathcal{E}_k^-)\le 16e^{-2C}\,.
	\end{equation}
	and the conclusion follows by \eqref{eq:idea-lb} and \eqref{eq:lower-bound-subcritical} by taking $C\to\infty$.
\end{proof}
We now prove the result for the phase in which $\gamma_\infty\ge 0$ and we first focus on the case in which $(2\beta-1)\log N\to\infty$.
\begin{proof}[Delayed cutoff]  For the upper bound, letting $t^+=\frac{1+\gamma}{2}\log N+C\alpha^{-1}=\frac{2\beta-1}{2\alpha}\log N + C\alpha^{-1}$ and using \eqref{eq:upper-bound} we get, for all $N$ sufficiently large,
	\begin{align*}
		\mathcal{D}_{\alpha,\textcolor{black}{m}}^{N,{\rm tv}}(t^+)\le& \frac12 \sqrt{\left[1+\frac14\left(N^{\beta-1-\beta+\frac12}e^{- C}+N^{-\frac12-\frac\gamma2}e^{-C\alpha^{-1}}\right)^2\right]^N-1} \\
		\le&\frac12\sqrt{\left[1+\frac{e^{-2 C}}{N}\right]^N-1} \le\sqrt{e^{e^{-2C}}-1}\ .
	\end{align*}
	For the lower bound, set $t^-=\frac{2\beta-1}{2\alpha}\log N-C\alpha^{-1}$ and notice that, \textcolor{black}{similarly to} \eqref{eq:eps},
	\begin{equation}
		\textcolor{black}{\,	\frac{N}{2}-\mathbb{E}[S(t^-)]\ge \frac{1}{2}m N^{\tfrac12-\beta}e^C\ge \frac14\sqrt{N}e^C}\,.
	\end{equation}
	 Therefore, the conclusion follows again by \eqref{eq:lower-bound-subcritical}.
\end{proof}
Finally, we consider the case in which $\gamma_\infty\ge 0$ and  $(2\beta-1)\log N\to \ell\in[0,\infty)$.

\begin{color}{black}
\begin{proof}[No cutoff] Let $C>0$ be a constant to be tuned later. Since $(2\beta-1)\log N\to \ell <\infty$, for all $\varepsilon>0$ we have
		\begin{equation}\label{eq:limit}
			2\beta-1<\frac{\ell+\varepsilon}{\log(N)}\,,
		\end{equation}
	for all $N$ sufficiently large. On the other hand, since $\gamma_\infty\ge 0$, we have $\lim_{N\to\infty}\frac{2\beta-1}{\alpha}\ge 1$ and by \eqref{eq:limit} we deduce that
	\begin{equation}\label{eq:beta-alpha}
		\alpha<\frac{\ell+\varepsilon}{\log(N)}\,,
	\end{equation}
for all $N$ large enough. By \eqref{eq:upper-bound}, and exploiting \eqref{eq:limit} and \eqref{eq:beta-alpha} we have: for all $\varepsilon>0$ and $N$ large enough
	\begin{equation}\label{eq:ub-super-gamma->infinity}
	\begin{split}
		\mathcal{D}_{\alpha,\textcolor{black}{m}}^{N,{\rm tv}}(C\alpha^{-1})\le& \frac12 \sqrt{\left[1+\frac14\left(N^{\beta-1}e^{- C}+Ne^{-C\alpha^{-1}}\right)^2\right]^N-1}\\ 
		\le& \frac12 \sqrt{\left[1+\frac14\left(N^{-\beta}e^{- C+\ell+\varepsilon}+N^{1-\frac{C}{\ell +\varepsilon}}\right)^2\right]^N-1}\,.
	\end{split}
\end{equation}
Choosing now $C>2(\ell+\varepsilon)$ we have
$$N^{-\beta}e^{- C+\ell+\varepsilon}>N^{1-\frac{C}{\ell +\varepsilon}}\,,$$
for all $N$ large enough, and therefore
	\begin{equation}\label{eq:ub-super-gamma->infinity-bis}
	\begin{split}
	\mathcal{D}_{\alpha,\textcolor{black}{m}}^{N,{\rm tv}}(C\alpha^{-1})
		\le& \frac12 \sqrt{\left[1+N^{-2\beta}e^{- 2(C-\ell-\varepsilon)}\right]^N-1}\\
		\le& \frac12 \sqrt{{\rm exp}(N^{-2\beta+1}e^{-2(C-\ell-\varepsilon)})-1}\le \frac12 \sqrt{e^{e^{-2(C-\ell-\varepsilon)-\ell+\varepsilon}}-1}\,,	
	\end{split}
\end{equation}
where in the last step we used that: for all $\varepsilon>0$
\begin{equation}\label{eq:limit-2}
	-(2\beta-1)<\frac{-\ell+\varepsilon}{\log(N)}\,,
\end{equation}
for all $N$ sufficiently large. In conclusion, for all $\varepsilon>0$ and $N$ large enough  
	\begin{equation}\label{eq:ub-super-gamma->infinity-tris}
	\begin{split}
		\mathcal{D}_{\alpha,\textcolor{black}{m}}^{N,{\rm tv}}(C\alpha^{-1})
		\le& \frac12 \sqrt{e^{e^{-2C+\ell+3\varepsilon}}-1}\,,	
	\end{split}
\end{equation}
and the latter is strictly smaller than $1$ as soon as $C=C(\ell)$ is sufficiently large.

For the lower bound we use the estimates in Section \ref{suse:CLT}. Indeed, for all $\varepsilon>0$ and $N$ large enough
	\begin{equation*}
		\frac{N}2-\mathbb{E}[S(C\alpha^{-1})]\ge \frac12 N^\beta e^{-C}\ge \frac12 \sqrt{N} e^{-C-\varepsilon}\,,
	\end{equation*}
where we used that
\begin{equation}
	2\beta-1> \frac{\ell-\varepsilon}{\log(N)}\qquad \Longrightarrow \qquad \beta> \frac12 - \frac{\varepsilon}{\log(N)}\,,
\end{equation}
for all $N$ sufficiently large. Therefore, by \eqref{eq:lower-bound-CLT-star-bis} and \eqref{eq:lower-bound-CLT-t-k-bis} with $ \textcolor{black}{c}=\frac12e^{-C-\varepsilon}$ and $k=\lceil\frac{ \textcolor{black}{c}}{2} \sqrt{N}\rceil$ we have: for all $\varepsilon,\delta>0$
	\begin{align*}
		\mathbb{P}(S(C\alpha^{-1})\not\in\mathcal{E}_k^-)\le1-\Phi(\tfrac12 e^{-C-\varepsilon})+\delta\,,\qquad \mathbb{P}(S^\star\in \mathcal{E}^-_k)\le \Phi(-\tfrac12 e^{-C-\varepsilon})+\delta\ ,
	\end{align*}
	so that, for all $C>0$ and every $N$ sufficiently large, thanks to \eqref{eq:idea-lb},
	\begin{equation}\label{eq:last}
		\mathcal{D}_{\alpha,\textcolor{black}{m}}^{N,{\rm tv}}(C\alpha^{-1})\ge \Phi(\tfrac12 e^{-C-\varepsilon})-\Phi(-\tfrac12 e^{-C-\varepsilon})-2\delta\,,
	\end{equation}
and, for any $C>0$, taking $\delta=\varepsilon=\varepsilon(C)$ sufficiently small, the right-hand side of \eqref{eq:last} is strictly larger than 0 for all $N$ large enough.
\end{proof}
\end{color}
This concludes the proof of Theorem \ref{th:tv}.
\section{Proof of Proposition \ref{prop:whole-chain} and Corollary \ref{coro}}\label{sec:proof-prop}
\begin{proof}[Proof of Proposition \ref{prop:whole-chain}]
\textcolor{black}{If $\tilde{\gamma}_\infty<0$ then $N-m\to\infty$}, hence the ``regular'' coordinate exhibits cutoff at $t^R$ with window of size $\Theta(1)$. Therefore, the lower bound in \eqref{eq:prop-insensitivity} follows by projection. Moreover, if \textcolor{black}{$m\to\infty$} the coordinate $(H_t)_{t\ge 0}$ exhibits cutoff at $t^H$ with a window of size $\Theta(\alpha^{-1})$. Hence, also the lower bound \eqref{eq:prop-delayed} follows by projection. \begin{color}{black}
To show the lower bound in \eqref{eq:prop-no-cutoff}, recall that $t_{\rm rel}=\alpha^{-1}$ and that in this case we have $\alpha\lesssim (\log N)^{-1}$.
	Then, by \cite[Lemma 20.12]{LP17}, at time $t=C\alpha^{-1}$ the ``heavy'' coordinate is at total-variation distance at least $\tfrac12 e^{-\frac{t}{t_{\rm rel}}}=\tfrac12e^{-C}$ from its stationary distribution, and therefore the desired lower bound follows again by projection.
\end{color}
	To get the upper bounds in \eqref{eq:prop-insensitivity} and \eqref{eq:prop-delayed} it is enough to argue as in the proof of \cite[Theorem 20.7]{LP17}: thanks to the product structure of the chain we have
	\begin{equation}\label{eq:LP}
		\max_{(r,h)\in\mathcal{X}}\| \chi^{r,h}_t-\chi\|_{\rm tv}^2\le 2 N^\beta e^{-2\alpha t}+  2 N e^{-2 t}\,.
	\end{equation}
	Hence, the upper bounds in \eqref{eq:prop-insensitivity}, \eqref{eq:prop-delayed} \textcolor{black}{ and \eqref{eq:prop-no-cutoff}} follow by choosing $t=t^R+C$, $t=t^H+C\alpha^{-1}$ \textcolor{black}{and $t=C\alpha^{-1}$}, respectively, for some $C=C(\varepsilon)$ large enough.
\end{proof}

\begin{color}{black}
\begin{proof}[Proof of Corollary \ref{coro}]
Regardless of the choice of the parameters we have
	\begin{equation}\label{eq:tmix}
		t_{\rm mix}\asymp t^R + t^H\asymp \log(N) +t^H\,,
	\end{equation}
	since the mixing time of the whole chain is always larger or equal than the mixing time of one of the two coordinates and smaller or equal than their sum. Using $t_{\rm rel}=\alpha^{-1}$ and \eqref{eq:tmix},  simple algebra shows that $\liminf t_{\rm mix}/t_{\rm rel}<\infty$ if and only if $\alpha+\beta\lesssim (\log N)^{-1}$ (at least along subsequences), and the latter implies $m\asymp 1$ and $\tilde{\gamma}_\infty=0$. Conversely, in the {\em insensitivity} regime the condition $\tilde{\gamma}_\infty$ ensures that $\alpha\asymp1$ and therefore $t_{\rm mix}/t_{\rm rel}\gg1 $; and, similarly, in the {\em delayed cutoff} regime the assumption $m\to\infty$ ensures that $\beta\gg (\log N)^{-1}$ and therefore $t_{\rm mix}/t_{\rm rel}\gg 1$.
\end{proof}
\end{color}

\begin{acks}  
The author is a member of GNAMPA-INdAM, and he thanks the German Research Foundation (project number 444084038, priority program SPP2265) for financial support. Moreover, the author wishes to thank Pietro Caputo and Federico Sau for helpful discussions on the topic, \textcolor{black}{and the anonymous referee for having suggested several key improvements on the first draft of this paper.}
\end{acks}


\end{document}